\documentclass[11pt]{article}
\usepackage{amsmath,amsfonts, amssymb,amsthm}
\bibstyle{plain}

\def\be{\begin{equation}}
\def\en{\end{equation}}
\def\bee{\begin{eqnarray*}}
\def\ene{\end{eqnarray*}}
\def\R{{\bf R}}
\def\R{{\mathbb R}}

\def\Var{{\rm Var}}
\def\Ent{{\rm Ent}}

\def\ep{\varepsilon}

\newtheorem{thm}{Theorem}
\numberwithin{thm}{section}
\newtheorem{cor}[thm]{Corollary}
\newtheorem{lem}[thm]{Lemma}
\newtheorem{rem}[thm]{Remark} 
\newtheorem{prop}[thm]{Proposition}

\theoremstyle{definition}
\newtheorem{ex}{Example}

\title{Concentration Properties of Restricted Measures with Applications to Non-Lipschitz Functions}
\author{S. G. Bobkov\thanks{School of Mathematics, University of Minnesota, Minneapolis, MN 55455}, P. Nayar\thanks{Institute of Mathematics \& Applications, Minneapolis, MN 55455; E-mail address: nayar@mimuw.edu.pl (corresponding author); research supported in part by NCN grant DEC-2012/05/B/ST1/00412}, and P. Tetali\thanks{School of Mathematics and School of Computer Science, Georgia Institute of Technology, Atlanta, GA 30332; research supported in part by NSF DMS-1407657.}}
\begin{document}

\maketitle


\vspace{0.3cm}

concentration of measure. 

\abstract{We show that for any metric probability space $(M,d,\mu)$ with a subgaussian constant $\sigma^2(\mu)$ and any set $A \subset M$ we have
$
	\sigma^2(\mu_A) \leq  c \log\left(e/\mu(A)\right)\,\sigma^2(\mu)
$, where $\mu_A$ is a restriction of $\mu$ to the set $A$ and $c$ is a universal constant. As a consequence we deduce concentration inequalities  for non-Lipschitz functions.
    }

\section{Introduction}\label{intro}
\setcounter{equation}{0}

It is known that many high-dimensional probability distributions $\mu$
on the Euclidean space $\R^n$ (and other metric spaces, including graphs) possess 
strong concentration properties. In a functional language, this may informally be
stated as the assertion that any sufficiently smooth function $f$ on $\R^n$, e.g.,
having a bounded Lipschitz semi-norm, is almost a constant on almost all space. 
There are several ways to quantify such a property. One natural approach proposed 
by N. Alon, R. Boppana and J. Spencer \cite{A-B-S} associates with a given metric 
probability space $(M,d,\mu)$ its {\em spread constant},
\[
s^2(\mu) = \sup\, \Var_\mu(f) = \sup \int (f-m)^2\,d\mu,
\]
where $m = \int f\,d\mu$, and the sup is taken over all functions $f$ on $M$ with 
$\|f\|_{\rm Lip} \leq 1$. More information is 
contained in the so-called subgaussian constant $\sigma^2 = \sigma^2(\mu)$ which is 
defined as the infimum over all $\sigma^2$ such that
\be \label{1.1}
\int e^{tf}\,d\mu \leq e^{\sigma^2 t^2/2}, \quad {\rm for \ all} \ \ t \in \R,
\en
for any $f$ on $M$ with $m=0$ and $\|f\|_{\rm Lip} \leq 1$ (cf. \cite{B-G-H}).
This quantity may also be introduced via the transport-entropy inequality relating 
the classical Kantorovich distance and the relative entropy from an arbitrary 
probability measure on $M$ to the measure $\mu$ (cf. \cite{B-G}).

While in general $s^2 \leq \sigma^2$, the latter characteristic allows one 
to control subgaussian tails under the probability 
measure $\mu$ uniformly in the entire class of Lipschitz functions on $M$.
More generally, when $\|f\|_{\rm Lip} \leq L$, \eqref{1.1} yields
\be \label{1.2}
\mu\{|f - m| \geq t\} \leq 2 e^{-t^2/(\sigma^2 L^2)}, \qquad t > 0.
\en

Classical and well-known examples include the standard Gaussian measure
on $M = \R^n$ in which case $s^2 = \sigma^2 = 1$, and the normalized 
Lebesgue measure on the unit sphere $M = S^{n-1}$ with 
$s^2 = \sigma^2 = \frac{1}{n-1}$. The last example was a starting point in the
study of the concentration of measure phenomena, a fruitful direction initiated 
in the early 1970s by V. D. Milman.

Other examples come often after verification that $\mu$ satisfies certain
Sobolev-type inequalities such as Poincar\'e-type inequalities
\[
\lambda_1 \Var_\mu(u) \leq \int |\nabla u|^2\,d\mu, 
\]
and logarithmic Sobolev inequalities
\[
\rho\, \Ent_\mu(u^2) = 
\rho \bigg[\int u^2 \log u^2\,d\mu - \int u^2\,d\mu\, \log \int u^2\,d\mu\bigg]
\leq 2 \int |\nabla u|^2\,d\mu,
\]
where $u$ may be any locally Lipschitz function on $M$, and the constants 
$\lambda_1 > 0$ and $\rho > 0$ do not depend on $u$. Here the modulus of the gradient 
may be understood in the generalized sense as the function
\[
|\nabla u(x)| = \limsup_{y \rightarrow x} \frac{|u(x) - u(y)|}{d(x,y)}, \qquad x \in M
\]
(this is the so-called ``continuous setting"), while in the discrete spaces, e.g., graphs, 
we deal with other naturally defined gradients.
In both cases, one has respectively the well-known upper bounds
\be \label{1.3}
s^2(\mu) \leq \frac{1}{\lambda_1}, \qquad \sigma^2(\mu) \leq \frac{1}{\rho}.
\en
For example, $\lambda_1 = \rho = n-1$ on the unit sphere (best possible values, \cite{M-W}), 
which can be used to make a corresponding statement about the spread and Gaussian
constants.

One of the purposes of this note is to give new examples by involving the family 
of the normalized restricted measures
\[
\mu_A(B) = \frac{\mu(A \cap B)}{\mu(A)}, \qquad B \subset M \ ({\rm Borel}),
\]
where a set $A \subset M$ is fixed and has a positive measure. As an example, 
returning to the standard Gaussian measure $\mu$ on $\R^n$, it is known that 
$\sigma^2(\mu_A) \leq 1$ for any convex body $A \subset \R^n$. This remarkable property, 
discovered by D. Bakry and M. Ledoux \cite{B-L} in a sharper form of a Gaussian-type 
isoperimetric inequality, has nowadays several proofs and generalizations, cf. \cite{B1,B2}.
Of course, in general, the set $A$ may have a rather disordered structure, for example, 
to be disconnected. And then there is no hope for validity of a Poincar\'e-type inequality
for the measure $\mu_A$.
Nevertheless, it turns out that the concentration property of $\mu_A$ is inherited
from $\mu$, unless the measure of $A$ is too small. In particular, 
we have the following observation about abstract metric probability spaces.

\begin{thm}\label{t.1.1}
For any measurable set $A \subset M$ with $\mu(A) > 0$,
the subgaussian constant $\sigma^2(\mu_A)$ of the normalized restricted measure satisfies
\be \label{1.4}
\sigma^2(\mu_A) \, \leq \, c\,\log\Big(\frac{e}{\mu(A)}\Big)\,\sigma^2(\mu),
\en
where $c$ is an absolute constant.
\end{thm}

One may further generalize this assertion by defining the subgaussian constant
$\sigma^2_{\cal F}(\mu)$ within a given fixed subclass $\cal F$ of functions 
on $M$, by using the same bound \eqref{1.1} on the Laplace transform. This is motivated by a 
possible different level of concentration for different classes; indeed, 
in case of $M = \R^n$, the concentration property may considerably be strengthened for the 
class $\cal F$ of all convex Lipschitz functions. In particular, one
result of M. Talagrand \cite{T1,T2} provides a dimension-free bound 
$\sigma^2_{\cal F}(\mu) \leq C$ for an arbitrary product probability measure $\mu$ 
on the $n$-dimensional cube $[-1,1]^n$. Hence, a more general 
version of Theorem \ref{t.1.1} yields the bound
\[
\sigma_{\cal F}^2(\mu_A) \, \leq \, c\,\log\Big(\frac{e}{\mu(A)}\Big)
\]
with some absolute constant $c$, which holds for any Borel subset $A$ of $[-1,1]^n$
(cf. Section \ref{deviations} below).

According to the very definition, the quantities $\sigma^2(\mu)$ and $\sigma^2(\mu_A)$ 
might seem to be responsible for deviations of only Lipschitz functions $f$ on $M$ and 
$A$, respectively. However, the inequality \eqref{1.4} may also be used to control deviations 
of non-Lipschitz $f$ -- on large parts of the space and under certain regularity hypotheses.
Assume, for example, $\int |\nabla f|\,d\mu \leq 1$ (which is kind of a normalization 
condition) and consider
\be \label{1.5}
A = \{x \in M: |\nabla f(x)| \leq L\}.
\en
If $L \geq 2$, this set has the measure $\mu(A) \geq 1 - \frac{1}{L} \geq \frac{1}{2}$, 
and hence, $\sigma^2(\mu_A) \leq c\sigma^2(\mu)$ with some absolute constant $c$.
If we assume that $f$ has a Lipschitz semi-norm $\leq L$ on $A$, then, according to \eqref{1.2},
\be \label{1.6}
\mu_A\{x \in A: |f - m| \geq t\} \leq 2 e^{-t^2/(c\sigma^2(\mu) L^2)}, \qquad t > 0,
\en
where $m$ is the mean of $f$ with respect to $\mu_A$.
It is in this sense one may say that $f$ is almost a constant on the set $A$. 

This also yields a corresponding deviation bound on the whole space,
\[
\mu\{x \in M: |f - m| \geq t\} \leq 2 e^{-t^2/c\sigma^2(\mu) L^2} + \frac{1}{L}.
\]
Stronger integrability conditions posed on $|\nabla f|$ can considerably sharpen 
the conclusion. By a similar argument, Theorem \ref{t.1.1} yields, for example, the following 
exponential bound, known in the presence of a logarithmic Sobolev inequality 
for the space $(M,d,\mu)$, and with $\sigma^2$ replaced by $1/\rho$ (cf. \cite{B-G}).

\begin{cor}\label{c.1.2}
Let $f$ be a locally Lipschitz function on $M$ with Lipschitz 
semi-norms $\leq L$ on the sets \eqref{1.5}. If $\int e^{|\nabla f|^2}\,d\mu \leq 2$, 
then $f$ is $\mu$-integrable, and moreover,
\[
\mu\{x \in M: |f - m| \geq t\} \leq 2 e^{-t/c\sigma(\mu)}, \qquad t > 0,
\]
where $m$ is the $\mu$-mean of $f$ and $c$ is an absolute constant.
\end{cor}

Equivalently (up to an absolute factor), we have a Sobolev-type inequality
\[
\|f - m\|_{\psi_1} \leq c\sigma(\mu)\, \| \nabla f \|_{\psi_2},
\]
connecting the $\psi_1$-norm of $f-m$ with the $\psi_2$-norm of the modulus of the
gradient of $f$. We prove a more general version of this corollary in Section \ref{deviations}
(cf. Theorem \ref{t.7.1}). As will be explained in the same section, similar assertions may 
also be made about convex $f$ and product measures $\mu$ on $M = [-1,1]^n$, 
thus extending Talagrand's theorem to the class of non-Lipschitz functions.

In view of the right bound in \eqref{1.3} and \eqref{1.4}, the spread and subgaussian constants 
for restricted measures can be controled in terms of the logarithmic Sobolev constant 
$\rho$ via
\[
s^2(\mu_A) \, \leq \, 
\sigma^2(\mu_A) \, \leq \, c\,\log\Big(\frac{e}{\mu(A)}\Big)\,\frac{1}{\rho}.
\]
However, it may happen that $\rho = 0$ and $\sigma^2(\mu) = \infty$, while
$\lambda_1 > 0$ (e.g., for the product exponential distribution on $\R^n$). 
Then one may wonder whether one can estimate the spread constant of a restricted measure
in terms of the spectral gap. In that case there is a bound similar to \eqref{1.4}.

\begin{thm}\label{t.1.3}
Assume the metric probability space $(M,d,\mu)$ satisfies
a Poincar\'e-type inequality with $\lambda_1>0$. For any
$A \subset M$ with $\mu(A) > 0$, with some absolute constant $c$
\be \label{1.7}
s^2(\mu_A) \, \leq \, c\,\log^2\Big(\frac{e}{\mu(A)}\Big)\, \frac{1}{\lambda_1}.
\en
\end{thm}

It should be mentioned that the logarithmic terms in \eqref{1.4} and \eqref{1.7} may not be
removed and are actually asymptotically optimal as functions of $\mu(A)$,
as $\mu(A)$ is getting small, see Section \ref{opt}.

Our contribution below is organized into sections as follows:

\vskip2mm

2. Bounds on $\psi_\alpha$-norms for restricted measures.

3. Proof of Theorem \ref{t.1.1}. Transport-entropy formulation.

4. Proof of Theorem \ref{t.1.3}. Spectral gap.

5. Examples.

6. Deviations for non-Lipschitz functions.

7. Optimality.

8. Appendix.

\section{Bounds on $\psi_\alpha$-norms for restricted measures}
\setcounter{equation}{0}

A measurable function $f$ on the probability space $(M,\mu)$ is said to have
a finite $\psi_\alpha$-norm, $\alpha \geq 1$, if for some $r > 0$,
\[
\int e^{(|f|/r)^\alpha}\,d\mu \leq 2.
\]
The infimum over all such $r$ represents the $\psi_\alpha$-norm $\|f\|_{\psi_\alpha}$ 
or $\|f\|_{L^{\psi_\alpha}(\mu)}$, which is just the Orlicz norm associated 
with the Young function $\psi_\alpha(t) = e^{|t|^\alpha} - 1$. 

We are mostly interested in the particular cases $\alpha = 1$ and $\alpha = 2$.
In this section we recall well-known relations between the $\psi_1$ and $\psi_2$-norms
and the usual $L^p$-norms  $\|f\|_p = \|f\|_{L^p(\mu)} = (\int |f|^p\,d\mu)^{1/p}$.
For the readers' convenience, we include the proof in the appendix.

\begin{lem}\label{l.2.1}
We have
\be \label{2.1}
\sup_{p \geq 1} \frac{\|f\|_p}{\sqrt{p}} \leq
\|f\|_{L^{\psi_2}(\mu)} \leq 4\, \sup_{p \geq 1} \frac{\|f\|_p}{\sqrt{p}},
\en
\be \label{2.2}
\sup_{p \geq 1} \frac{\|f\|_p}{p} \leq
\|f\|_{L^{\psi_1}(\mu)} \leq 6\, \sup_{p \geq 1} \frac{\|f\|_p}{p}.
\en
\end{lem}

Given a measurable subset $A$ of $M$ with $\mu(A) > 0$, we consider
the normalized restricted measure $\mu_A$ on $M$, i.e.,
\[
\mu_A(B) = \frac{\mu(A \cap B)}{\mu(A)}, \qquad B \subset M.
\]
Our basic tool leading to Theorem \ref{t.1.1} will be the following assertion.

\begin{prop}\label{p.3.1}
For any measurable function $f$ on $M$,
\be \label{3.1}
\|f\|_{L^{\psi_2}(\mu_A)} \, \leq \, 4e\,\log^{1/2}\Big(\frac{e}{\mu(A)}\Big)\, 
\|f\|_{L^{\psi_2}(\mu)}.
\en
\end{prop}

\begin{proof} Assume that $\|f\|_{L^{\psi_2}(\mu)} = 1$
and fix $p \geq 1$. By the left inequality in \eqref{2.1}, for any $q \geq 1$,
\[
q^{q/2} \geq \int |f|^q\,d\mu \geq \mu(A) \int |f|^q\,d\mu_A,
\]
so
\[
\frac{\|f\|_{L^q(\mu_A)}}{\sqrt{q}} \leq \bigg(\frac{1}{\mu(A)}\bigg)^{1/q}.
\]
But by the right inequality in \eqref{2.1},
\[
\|f\|_{\psi_2} \, \leq \,
4\, \sup_{q \geq 1} \frac{\|f\|_q}{\sqrt{q}} \, \leq \,
4 \sqrt{p}\, \sup_{q \geq p} \frac{\|f\|_q}{\sqrt{q}}.
\]
Applying it on the space $(M,\mu_A)$, we then get
\bee
\|f\|_{L^{\psi_2}(\mu_A)} 
 & \leq &
4 \sqrt{p}\, \sup_{q \geq p} \frac{\|f\|_{L^q(\mu_A)}}{\sqrt{q}} \\ 
 & \leq &
4 \sqrt{p}\,\sup_{q \geq p}\, \bigg(\frac{1}{\mu(A)}\bigg)^{1/q} \, = \,
4 \sqrt{p}\, \bigg(\frac{1}{\mu(A)}\bigg)^{1/p}.
\ene
The obtained inequality,
\[
\|f\|_{L^{\psi_2}(\mu_A)} \, \leq \, 4 \sqrt{p}\,
\bigg(\frac{1}{\mu(A)}\bigg)^{1/p},
\]
holds true for any $p \geq 1$ and therefore may be optimized over $p$.
Choosing $p = \log\frac{e}{\mu(A)}$, we arrive at \eqref{3.1}.
\end{proof}

A possible weak point in the bound \eqref{3.1} is that the means of $f$ are not involved.
For example, in applications, if $f$ were defined only on $A$ and had $\mu_A$-mean zero, 
we might need to find an extension of $f$ to the whole space $M$ keeping the mean zero 
with respect to $\mu$. In fact, this should not create any difficulty, since one may
work with the symmetrization of $f$. 

More precisely, we may apply Proposition \ref{p.3.1} on the product space 
$(M \times M, \mu \otimes \mu)$ to the product sets
$A \times A$ and functions of the form $f(x) - f(y)$. Then we get
\[
\|f(x)-f(y)\|_{L^{\psi_2}(\mu_A \otimes \mu_A)} \, \leq \,
4e\,\log^{1/2}\bigg(\frac{e}{\mu(A)^2}\bigg)\, 
\|f(x) - f(y)\|_{L^{\psi_2}(\mu \otimes \mu)}.
\]
Since
$
\log\big(\frac{e}{\mu(A)^2}\big) \leq 2\log\big(\frac{e}{\mu(A)}\big),
$
we arrive at:

\begin{cor}\label{c.3.2}
For any measurable function $f$ on $M$,
\[
\|f(x)-f(y)\|_{L^{\psi_2}(\mu_A \otimes \mu_A)} \, \leq \,
4e\sqrt{2}\,\log^{1/2}\Big(\frac{e}{\mu(A)}\Big)\, 
\|f(x) - f(y)\|_{L^{\psi_2}(\mu \otimes \mu)}.
\]
\end{cor}

Let us now derive an analog of Proposition \ref{p.3.1} for the $\psi_1$-norm, 
using similar arguments. Assume that $\|f\|_{L^{\psi_1}(\mu)} = 1$
and fix $p \geq 1$. By the left inequality in \eqref{2.2}, for any $q \geq 1$,
\[
q^q \geq \int |f|^q\,d\mu \geq \mu(A) \int |f|^q\,d\mu_A,
\]
so
\[
\frac{\|f\|_{L^q(\mu_A)}}{q} \leq \Big(\frac{1}{\mu(A)}\Big)^{1/q}.
\]
But, by the inequality \eqref{2.2},
\[
\|f\|_{L^{\psi_1}} \, \leq \, 6\, \sup_{q \geq 1} \frac{\|f\|_q}{q} \, \leq \,
6p\, \sup_{q \geq p} \frac{\|f\|_q}{q}.
\]
Applying it on the space $(M,\mu_A)$, we get
\bee
\|f\|_{L^{\psi_1}(\mu_A)} 
 & \leq &
6p\, \sup_{q \geq p} \frac{\|f\|_{L^q(\mu_A)}}{q} \\ 
 & \leq &
6p\,\sup_{q \geq p}\, \Big(\frac{1}{\mu(A)}\Big)^{1/q} \, = \,
6p\,\Big(\frac{1}{\mu(A)}\Big)^{1/p}.
\ene
The obtained inequality,
\[
\|f\|_{L^{\psi_1}(\mu_A)} \, \leq \, 6p\,\Big(\frac{1}{\mu(A)}\Big)^{1/p},
\]
holds true for any $p \geq 1$ and therefore may be optimized over $p$.
Choosing $p = \log\frac{e}{\mu(A)}$, we arrive at:

\begin{prop}\label{p.3.3}
For any measurable function $f$ on $M$, we have
\[
\|f\|_{L^{\psi_1}(\mu_A)} \, \leq \, 6e\,\log\Big(\frac{e}{\mu(A)}\Big)\, 
\|f\|_{L^{\psi_1}(\mu)}.
\]
\end{prop}

Similarly to Corollary \ref{c.3.2} one may write down this relation on the product 
probability space $(M \times M,\mu \otimes \mu)$ with the functions of the form 
$\tilde f(x,y) = f(x)-f(y)$ and the product sets $\tilde A = A \times A$. 
Then we get
\be \label{3.2}
\|f(x) - f(y)\|_{L^{\psi_1}(\mu_A \otimes \mu_A)} \, \leq \, 
12\,e\,\log\Big(\frac{e}{\mu(A)}\Big)\, \|f(x) - f(y)\|_{L^{\psi_1}(\mu \otimes \mu)}.
\en

\section{Proof of Theorem \ref{t.1.1}. Transport-entropy formulation}
\setcounter{equation}{0}

The finiteness of the subgaussian constant for a given metric probability space 
$(M,d,\mu)$ means that $\psi_2$-norms of Lipschitz functions on $M$ with mean zero 
are uniformly bounded. Equivalently, for any (for all) $x_0 \in M$, we have that,
for some $\lambda>0$,
\[
\int e^{d(x,x_0)^2/\lambda^2}\,d\mu(x) < \infty.
\]
The definition \eqref{1.1} of $\sigma^2(\mu)$ inspires to consider another norm-like quantity
\[
\sigma^2_f = \sup_{t \neq 0} \bigg[\frac{1}{t^2/2} \log\int e^{tf} \,d\mu\bigg].
\]
Here is a well-known relation (with explicit numerical constants) which holds 
in the setting of an abstract probability space $(M,\mu)$. Once again, we include a proof in the appendix for completeness.

\begin{lem}\label{l.4.1}
If $f$ has mean zero and finite $\psi_2$-norm, then
\[
\frac{1}{\sqrt{6}}\, \|f\|_{\psi_2}^2 \leq \sigma^2_f \leq 4\, \|f\|_{\psi_2}^2.
\]
\end{lem}

One can now relate the subgaussian constant of the restricted measure to the
subgaussian constant of the original measure. Let now $(M,d,\mu)$ be a metric
probability space. First, Lemma \ref{l.4.1} immediately yields an equivalent
description in terms of $\psi_2$-norms, namely
\be \label{4.1}
\frac{1}{\sqrt{6}}\, \sup_f \|f\|_{\psi_2}^2 \leq 
\sigma^2(\mu) \leq 4\, \sup_f \|f\|_{\psi_2}^2,
\en
where the supremum is running over all $f:M \rightarrow \R$ with $\mu$-mean zero
and $\|f\|_{\rm Lip} \leq 1$.
Here, one can get rid of the mean zero assumption by considering functions of the form 
$f(x) - f(y)$ on the product space $(M \times M, \mu \otimes \mu)$.
If $f$ has mean zero, then, by Jensen's inequality,
\[
\int\!\!\int e^{(f(x) - f(y))^2/r^2}\,d\mu(x)\,d\mu(y) \geq 
\int e^{f(x)^2/r^2}\,d\mu(x),
\]
which implies that
\[
\|f(x) - f(y)\|_{L^{\psi_2}(\mu \otimes \mu)} \geq \|f\|_{L^{\psi_2}(\mu)}.
\]
On the other hand, by the triangle inequality,
\[
\|f(x) - f(y)\|_{L^{\psi_2}(\mu \otimes \mu)} \leq 2\,\|f\|_{L^{\psi_2}(\mu)}.
\]
Hence, we arrive at another, more flexible relation, where the mean zero assumption 
may be removed.

\begin{lem}\label{l.4.2}
We have
\[
\frac{1}{4\sqrt{6}}\, \sup_{f \in \cal F} \|f(x) - f(y)\|_{L^{\psi_2}(\mu \otimes \mu)}^2 
\leq \sigma^2(\mu) \leq 
4\, \sup_{f \in \cal F} \|f(x) - f(y)\|_{L^{\psi_2}(\mu \otimes \mu)}^2,
\]
where the supremum is running over all functions $f$ on $M$ with $\|f\|_{\rm Lip} \leq 1$.
\end{lem}

\begin{proof}[Proof of Theorem \ref{t.1.1}.]
We are prepared to make last steps for the proof of the inequality \eqref{1.4}. 
We use the well-known Kirszbraun's theorem: Any function $f: A \rightarrow \R$ with 
Lipschitz semi-norm $\|f\|_{\rm Lip} \leq 1$ on $A$ admits a Lipschitz extension to 
the whole space (\cite{K}, \cite{MS}). Namely, one may put
\[
\tilde f(x) = \inf_{a \in  A} \big[f(a) + d(a,x)\big], \quad x \in M.
\]
Applying first Corollary \ref{c.3.2} and then the left inequality of Lemma \ref{l.4.2} to $\tilde f$, 
we get
\bee
\|f(x) - f(y)\|_{L^{\psi_2}(\mu_A \otimes \mu_A)}^2 
 & = & 
\big\|\tilde f(x) - \tilde f(y)\big\|_{L^{\psi_2}(\mu_A \otimes \mu_A)}^2 \\
 & \leq & 
\big(4e\sqrt{2}\,\big)^2\,\log\Big(\frac{e}{\mu(A)}\Big)\, 
\big\|\tilde f(x) - \tilde f(y)\big\|_{L^{\psi_2}(\mu \otimes \mu)}^2 \\ 
 & \leq & 
\big(4e\sqrt{2}\,\big)^2\,\log\Big(\frac{e}{\mu(A)}\Big)\cdot 
\big(4\sqrt{6}\,\big)^2\,\sigma^2(\mu).
\ene
Another application of Lemma \ref{l.4.2} -- in the space $(A,d,\mu_A)$ (now the right inequality) 
yields
\[
\sigma^2(\mu_A) \, \leq \, 4 \cdot 
\big(4e\sqrt{2}\,\big)^2\,\log\Big(\frac{e}{\mu(A)}\Big)\cdot 
\big(4\sqrt{6}\,\big)^2\,\sigma^2(\mu).
\]
This is exactly \eqref{1.4} with constant
$c = 4 \cdot (4e\sqrt{2}\,)^2\,(4\sqrt{6}\,)^2 = 3 \cdot 2^{12} e^2 = 90796.72...$
\end{proof}

\begin{rem}\label{r.4.3}
Let us also record the following natural generalization of Theorem
\ref{t.1.1}, which is obtained along the same arguments.
Given a collection $\cal F$ of (integrable) functions on the probability space 
$(M,\mu)$, define $\sigma_{\cal F}^2(\mu)$ as the infimum over all $\sigma^2$ such that
\[
\int e^{t(f-m)}\,d\mu \leq e^{\sigma^2 t^2/2}, \quad {\rm for \ all} \ \ t \in \R,
\]
for any $f \in \cal F$, where $m = \int f\,d\mu$. Then with the same constant $c$ as in
Theorem \ref{t.1.1}, for any measurable $A \subset M$, $\mu(A) > 0$, we have
\[
\sigma_{{\cal F}_A}^2(\mu_A) \, \leq \, 
c\,\log\Big(\frac{e}{\mu(A)}\Big)\,\sigma_{\cal F}^2(\mu),
\]
where ${\cal F}_A$ denotes the collection of restrictions of functions $f$ from $\cal F$
to the set $A$.
\end{rem}

Let us now mention an interesting connection of the subgaussian constants with the
Kantorovich distances
\[
W_1(\mu,\nu) = \inf \int\!\!\!\int d(x,y)\,\pi(x,y)
\]
and the relative entropies
\[
D(\nu||\mu) = \int \log\frac{d\nu}{d\mu}\,d\nu
\]
(called also Kullback-Leibler's distances or informational divergences).
Here, $\nu$ is a probability measure on $M$, which is absolutely continuous with 
respect to $\mu$ (for short, $\nu <\!< \mu$), and the infimum in the definition of 
$W_1$ is running over all probability measures $\pi$ on the product space $M \times M$ 
with marginal distributions $\mu$ and $\nu$, i.e., such that
\[
\pi(B \times M) = \mu(B), \quad \pi(M \times B) = \nu(B) \qquad 
({\rm Borel} \ B \subset M).
\]

As was shown in \cite{B-G}, if $(M,d)$ is a Polish space (complete separable), the subgaussian
constant $\sigma^2 = \sigma^2(\mu)$ may be described as an optimal value in the 
transport-entropy inequality
\be \label{4.2}
W_1(\mu,\nu) \leq \sqrt{2\sigma^2 D(\nu||\mu)}.
\en
Hence, we obtain from the inequality \eqref{1.4} a similar relation for measures $\nu$
supported on given subsets of $M$.

\begin{cor}\label{c.4.4}
Given a Borel probability measure $\mu$ on a Polish space 
$(M,d)$ and a closed set $A$ in $M$ such that $\mu(A) > 0$, for any Borel probability 
measure $\nu$ supported on $A$,
\[
W_1^2(\mu_A,\nu) \leq c\sigma^2(\mu) \log\Big(\frac{e}{\mu(A)}\Big)\,D(\nu||\mu_A),
\]
where $c$ is an absolute constant.
\end{cor}

This assertion is actually equivalent to Theorem \ref{t.1.1}.
Note that, for $\nu$ supported on $A$, there is an identity 
$D(\nu||\mu_A) = \log \mu(A) + D(\nu||\mu)$. In particular, 
$D(\nu||\mu_A) \leq D(\nu||\mu)$, so the relative entropies decrease when turning 
to restricted measures.

\section{Proof of Theorem \ref{t.1.3}. Spectral gap}
\setcounter{equation}{0}

Theorem \ref{t.1.1} insures, in particular, that, for any function $f$ on the metric
probability space $(M,d,\mu)$ with Lipschitz semi-norm $\|f\|_{\rm Lip} \leq 1$,
\[
\Var_{\mu_A}(f) \leq c\,\log\bigg(\frac{e}{\mu(A)}\bigg)\,\sigma^2(\mu)
\]
up to some absolute constant $c$. In fact, in order to reach a similar concentration 
property of the restricted measures, it is enough to start with a Poincar\'e-type
inequality on $M$,
\[
\lambda_1 \Var_\mu(f) \leq \int |\nabla f|^2\,d\mu.
\]
Under this hypothesis, a well-known theorem due to Gromov-Milman and Borovkov-Utev
asserts that mean zero Lipschitz functions $f$ have bounded 
$\psi_1$-norm. One may use a variant of this theorem proposed by Aida and Strook \cite{A-S},
who showed that
\[
\int e^{\sqrt{\lambda_1}\, f}\,d\mu \leq K_0 = 1.720102... \qquad (\|f\|_{\rm Lip} \leq 1).
\]
Hence
\[
\int e^{\sqrt{\lambda_1}\, |f|}\,d\mu \leq 2K_0 \quad {\rm and} \quad
\int e^{\frac{1}{2}\sqrt{\lambda_1}\, |f|}\,d\mu \leq \sqrt{2K_0} < 2,
\]
thus implying that $\|f\|_{\psi_1} \leq \frac{2}{\sqrt{\lambda_1}}$.
In addition,
\[
\int e^{\sqrt{\lambda_1}\, (f(x)-f(y))}\,d\mu(x) d\mu(y) \leq K_0^2, \qquad
\int e^{\sqrt{\lambda_1}\, |f(x)-f(y)|}\,d\mu(x) d\mu(y) \leq 2 K_0^2 < 6.
\]
From this,
\[
\int e^{\frac{1}{3}\sqrt{\lambda_1}\, |f(x)-f(y)|}\,d\mu(x) d\mu(y) < 6^{1/3} < 2,
\]
which means that $\|f(x) - f(y)\|_{\psi_1} \leq \frac{3}{\sqrt{\lambda_1}}$
with respect to the product measure $\mu \otimes \mu$ on the product space
$M \times M$. This inequality is translation invariant, so 
the mean zero assumption may be removed. Thus, we arrive at:

\begin{lem}\label{l.5.1}
Under the Poincar\'e-type inequality with spectral gap $\lambda_1>0$,
for any mean zero function $f$ on $(M,d,\mu)$ with $\|f\|_{\rm Lip} \leq 1$,
\[
\|f\|_{\psi_1} \leq \frac{2}{\sqrt{\lambda_1}}.
\]
Moreover, for any $f$ with $\|f\|_{\rm Lip} \leq 1$,
\be \label{5.1}
\|f(x)-f(y)\|_{L^{\psi_1}(\mu \otimes \mu)} \leq \frac{3}{\sqrt{\lambda_1}}.
\en
\end{lem}

This is a version of the concentration of measure phenomenon (with exponential 
integrability) in presence of a Poincar\'e-type inequality. Our goal is therefore
to extend this property to the normalized restricted measures $\mu_A$.
This can be achieved by virtue of the inequality \eqref{3.2} which when combined with
\eqref{5.1} yields an upper bound
\[
\|f(x) - f(y)\|_{L^{\psi_1}(\mu_A \otimes \mu_A)} \, \leq \, 
36\,e\,\log\bigg(\frac{e}{\mu(A)}\bigg)\, \frac{1}{\sqrt{\lambda_1}}.
\]
Moreover, if $f$ has $\mu_A$-mean zero, the left norm dominates 
$\|f\|_{L^{\psi_1}(\mu_A)}$ (by Jensen's inequality). We can summarize, taking into account once again Kirszbraun's theorem, as we did
in the proof of Theorem \ref{t.1.1}.

\begin{prop}\label{p.5.2}
Assume the metric probability space $(M,d,\mu)$ satisfies
a Poincar\'e-type inequality with constant $\lambda_1>0$. Given a measurable set 
$A \subset M$ with $\mu(A) > 0$, for any function $f:A \rightarrow \R$
with $\mu_A$-mean zero and such that $\|f\|_{\rm Lip} \leq 1$ on $A$,
\[
\|f\|_{L^{\psi_1}(\mu_A)} \, \leq \, 
36\,e\,\log\Big(\frac{e}{\mu(A)}\Big)\, \frac{1}{\sqrt{\lambda_1}}.
\]
\end{prop}

Theorem \ref{t.1.3} is now easily obtained with constant $c = 2\,(36e)^2$ by noting that 
$L^2$-norms are dominated by $L^{\psi_1}$-norms. More precisely, since 
$e^{|t|} - 1 \geq \frac{1}{2}\,t^2$, one has $\|f\|_{\psi_1}^2 \geq \frac{1}{2}\,\|f\|_2^2$.

\section{Examples}\label{examples}
\setcounter{equation}{0}

Theorems \ref{t.1.1} and \ref{t.1.3} involve a lot of interesting examples. Here are a few obvious cases.

\vskip3mm
\noindent  {\bf 1)} The standard Gaussian measure $\mu = \gamma$ on $\R^n$ satisfies a logarithmic 
Sobolev inequality on $M = \R^n$ with a dimension-free constant $\rho = 1$. 
Hence, from Theorem \ref{t.1.1} we get:

\begin{cor}\label{c.6.1}
For any measurable set $A \subset \R^n$ with $\gamma(A) > 0$,
the subgaussian constant $\sigma^2(\gamma_A)$ of the normalized restricted measure
$\gamma_A$ satisfies
\[
\sigma^2(\gamma_A) \, \leq \, c\,\log\Big(\frac{e}{\gamma(A)}\Big),
\]
where $c$ is an absolute constant. 
\end{cor}

As it was already mentioned, if $A$ is convex, there is a sharper bound
$\sigma^2(\gamma_A) \leq 1$. However, it may not hold without convexity assumption. 
Neverteless, if $\gamma(A)$ is bounded away from zero, we obtain a more universal 
principle.

Clearly, Corollary \ref{c.6.1} extends to all product measures $\mu = \nu^n$ on $\R^n$
such that $\nu$ satisfies a logarithmic Sobolev inequality on the real line, and with 
constants $c$ depending on $\rho$, only. A characterization of the property $\rho > 0$
in terms of the distribution function of the measure $\nu$ and the density of its 
absolutely continuous component may be found in \cite{B-G}.

\vskip3mm
\noindent {\bf 2)} Consider a uniform distribution $\nu$ on the shell
\[
A_\ep = \big\{x \in \R^n: 1 - \ep \leq |x| \leq 1\big\}, \qquad 0 \leq \ep \leq 1 \ \ 
(n \geq 2).
\]

\begin{cor}\label{c.6.2}
The subgaussian constant of $\nu$ satisfies\,
$\sigma^2(\nu) \leq \frac{c}{n}$, up to some absolute constant $c$.
\end{cor}

In other words, mean zero Lipschitz functions $f$ on $A_\ep$ are such that
$\sqrt{n}\,f$ are subgaussian. This property is well-known in the extreme cases -- 
on the unit Euclidean ball $A = B_n$ $(\ep = 1)$ and on the unit sphere 
$A = S^{n-1}$ $(\ep = 0)$.

Let $\mu$ denote the normalized Lebesgue measure on $B_n$.
In the case $\ep \geq \frac{1}{n}$, the shell $A_\ep$ represents the part of $B_n$ of 
measure 
\[
\mu(A_\ep) = 1 - \Big(1 - \frac{1}{n}\Big)^{n} \geq 1 - \frac{1}{e}. 
\]
Since the logarithmic Sobolev constant of the unit ball is of order $\frac{1}{n}$, 
and therefore $\sigma^2(\mu) \leq \frac{c}{n}$, the assertion of Corollary \ref{c.6.2} 
immediately follows from Theorem \ref{t.1.1}. If $\ep \leq \frac{1}{n}$, 
the assertion follows from a similar concentration property of the uniform 
distribution on the unit sphere. Indeed, with every Lipschitz function $f$ on $A_\ep$ 
one may associate its restriction to $S^{n-1}$, which is also Lipschitz (with respect 
to the Euclidean distance). On the other hand, for any $r \in [1-\ep,1]$ and 
$\theta \in S^{n-1}$, we have 
$|f(r\theta) - f(\theta)| \leq |r-1| \leq \ep \leq \frac{1}{n}$, thus proving the claim.

\vskip3mm
\noindent
{\bf  3)} The two-sided product exponential measure $\mu$ on $\R^n$ with density 
$2^{-n}\,e^{-(|x_1| + \dots + |x_n|)}$ satisfies a Poincar\'e-type
inequality on $M = \R^n$ with a dimension-free constant $\lambda_1 = 1/4$. Hence, 
from Proposition \ref{p.5.2} we get:

\begin{cor}\label{c.6.3}
For any measurable set $A \subset \R^n$ with $\mu(A) > 0$,
and for any function $f:A \rightarrow \R$ with $\mu_A$-mean zero and 
$\|f\|_{\rm Lip} \leq 1$, we have
\[
\|f\|_{L^{\psi_1}(\mu_A)} \, \leq \, c\,\log\Big(\frac{e}{\mu(A)}\Big),
\]
where $c$ is an absolute constant. In particular,
\[
s^2(\mu_A) \, \leq \, c\,\log^2\Big(\frac{e}{\mu(A)}\Big).
\]
\end{cor}

Clearly, Corollary 6.3 extends to all product measures $\mu = \nu^n$ on $\R^n$
such that $\nu$ satisfies a Poincar\'e-type inequality on the real line, and with 
constants $c$ depending on $\lambda_1$, only. A characterization of the property 
$\lambda_1 > 0$ may also be given in terms of the distribution function of $\nu$ 
and the density of its absolutely continuous component (cf. \cite{B-G}).

\vskip3mm
\noindent
{\bf 4a)} Let us take the metric probability space $(\{0,1\}^n,d_n,\mu)$, where $d_n$ 
is the Hamming distance, that is, $d_n(x,y)=\sharp \{i: x_i \ne y_i\}$, equipped with 
the uniform measure $\mu$. For this particular space, Marton established the 
transport-entropy inequality \eqref{4.2} with an optimal constant $\sigma^2 = \frac{n}{4}$,
cf. \cite{Mar}. Using the relation \eqref{4.2} as an equivalent 
definition of the subgaussian constant, we obtain from Theorem \ref{t.1.1}:

\begin{cor}\label{c.6.4}
For any non-empty set $A \subset \{0,1\}^n$,
the subgaussian constant $\sigma^2(\mu_A)$ of the normalized restricted measure
$\mu_A$ satisfies, up to an absolute constant $c$,
\be \label{6.1}
\sigma^2(\mu_A) \, \leq \, cn\,\log\Big(\frac{e}{\mu(A)}\Big).
\en
\end{cor}

\vskip3mm
\noindent
{\bf 4b)}
Let us now assume that $A$ is {\em monotone}, i.e., $A$ satisfies the condition 
\[
(x_1,\ldots,x_n) \in A \quad \implies \quad   (y_1,\ldots,y_n) \in A, \ \textrm{whenever} \ \ y_i \geq x_i, \ i=1,\ldots,n. 
\]
Recall that the discrete cube can be equipped with a natural graph structure: there is an edge between $x$ and $y$ whenever they are of Hamming distance $d_n(x,y)=1$.  For monotone sets $A$,  the graph metric $d_A$ on the subgraph on $A$ is equal to the restriction of $d_n$ to $A \times A$. Indeed, we have: 
\[
	d_n(x,y) \leq d_A(x,y) \leq d_A(x,x \wedge y) + d_A(y,x \wedge y) = d_n(x,x \wedge y) + d_n(y,x \wedge y) =d_n(x,y),
\]
where $x \wedge y = (x_1 \wedge y_1,\ldots,x_n \wedge y_n)$. Thus, 
\[
	s^2(\mu_A,d_A) \leq \sigma^2(\mu_A,d_A) \leq cn \log\left( \frac{e}{\mu(A)} \right).  
\]
This can be compared with what follows from a recent result of Ding and Mossel (see \cite{D-M}). The authors proved that the conductance (Cheeger constant) of $(A,\mu_A)$ satisfies $\phi(A) \geq \frac{\mu(A)}{16n}$. However, this type of isoperimetric results may not imply sharp concentration bounds. Indeed, by using Cheeger inequality, the above inequality leads to $\lambda_1 \geq c \mu(A)^2 /n^2$ and $s^2(\mu_A,d_A) \leq {1}/{\lambda_1} \leq c n^2/\mu(A)^2$, which is even worse than the trivial estimate $s^2(\mu_A,d_A) \leq \frac12 \textrm{diam}(A)^2 \leq n^2/2$.

\vskip3mm
\noindent
{\bf 5)} Let $(M,d,\mu)$ be a (separable) metric probability space with finite 
subgaussian constant $\sigma^2(\mu)$. The previous example can be naturally 
generalized to the product space $(M^n,\mu^n)$, when it is equipped with the 
$\ell^1$-type metric
\[
d_n(x,y) = \sum_{i=1}^n d(x_i,y_i), \qquad 
x = (x_1, \dots, x_n), \  y = (y_1, \dots, y_n) \in M^n.
\]
This can be done with the help of the following elementary observation.

\begin{prop}\label{p.6.5}
The subgaussian constant of the space $(M^n,d_n,\mu^n)$ 
is related to the subgaussian constant of $(M,d,\mu)$ by the equality
$
\sigma^2(\mu^n) = n\sigma^2(\mu).
$
\end{prop}

Indeed, one may argue by induction on $n$. Let $f$ be a function on $M^n$. The Lipschitz
property $\|f\|_{\rm Lip} \leq 1$ with respect to $d_n$ is equivalent to the assertion 
that $f$ is coordinatewise Lipschitz, that is, any function of the form 
$x_i \rightarrow f(x)$ has a Lipschitz semi-norm $\leq 1$ on $M$ for all fixed 
coordinates $x_j \in M$ ($j \neq i$). Hence, in this case, for all $t \in \R$,
\[
\int_M e^{tf(x)}\, d\mu(x_n) \leq 
\exp\Big\{t\int_M f(x)\, d\mu(x_n) + \frac{\sigma^2 t^2}{2}\Big\},
\]
where $\sigma^2 = \sigma^2(\mu)$. Here the function 
$(x_1,\dots,x_{n-1}) \rightarrow \int_M f(x)\, d\mu(x_n)$ is also coordinatewise
Lipschitz. Integrating the above inequality with respect to $d\mu^{n-1}(x_1,\dots,x_{n-1})$ 
and applying the induction hypothesis, we thus get 
\[
\int_{M^n} e^{tf(x)}\, d\mu^n(x) \leq \exp\Big\{t\int_{M^n} f(x)\, d\mu^n(x) + 
n\, \frac{\sigma^2 t^2}{2}\Big\}.
\]
But this means that $\sigma^2(\mu^n) \leq n \sigma^2(\mu)$.

For an opposite bound, it is sufficient to test \eqref{1.1} for $(M^n,d_n,\mu^n)$ 
in the class of all coordinatewise Lipschitz functions of the form
$f(x) = u(x_1) + \dots + u(x_n)$ with $\mu$-mean zero functions $u$ on $M$ such that
$\|u\|_{\rm Lip} \leq 1$.

\begin{cor}\label{c.6.6}
For any Borel set $A \subset M^n$ such that $\mu^n(A) > 0$, 
the subgaussian constant of the normalized restricted measure $\mu^n_A$ with respect 
to the $\ell^1$-type metric $d_n$ satisfies 
\[
\sigma^2(\mu^n_A) \leq cn \sigma^2(\mu)\, \log\Big(\frac{e}{\mu^n(A)}\Big),
\]
where $c$ is an absolute constant.
\end{cor}

For example, if $\mu$ is a probability measure on $M = \R$ such that
$\int_{-\infty}^\infty e^{x^2/\lambda^2}\,d\mu(x) \leq 2$ ($\lambda > 0$), then 
for the restricted product measures we have
\be \label{6.2}
\sigma^2(\mu^n_A) \, \leq \, cn \lambda^2\, \log\Big(\frac{e}{\mu^n(A)}\Big)
\en
with respect to the $\ell^1$-norm $\|x\|_1 = |x_1| + \dots + |x_n|$ on $\R^n$.

Indeed, by the integral hypothesis on $\mu$, for any $f$ on $\R$ with
$\|f\|_{\rm Lip} \leq 1$,
\bee
\int_{-\infty}^\infty \int_{-\infty}^\infty e^{(f(x)-f(y))^2/2\lambda^2}\,d\mu(x)d\mu(y)
 & \leq &
\int_{-\infty}^\infty \int_{-\infty}^\infty e^{(x-y)^2/2\lambda^2}\,d\mu(x)d\mu(y) \\ 
 & \leq &
\int_{-\infty}^\infty \int_{-\infty}^\infty e^{(x^2 + y^2)/\lambda^2}\,d\mu(x)d\mu(y) 
 \, \leq \, 4.
\ene
Hence, if $f$ has $\mu$-mean zero, by Jensen's inequality,
\[
\int_{-\infty}^\infty \int_{-\infty}^\infty e^{f(x)^2/4\lambda^2}\,d\mu(x) \, \leq \,
\int_{-\infty}^\infty \int_{-\infty}^\infty e^{(f(x)-f(y))^2/4\lambda^2}\,d\mu(x)d\mu(y)
 \, \leq \, 2,
\]
meaning that $\|f\|_{L^{\psi_2}(\mu)} \leq 2\lambda$. By Lemma \ref{l.4.1}, cf. \eqref{4.1}, it follows
that $\sigma^2(\mu) \leq 16 \lambda^2$, so, \eqref{6.2} holds true by an application of
Corollary \ref{c.6.6}.

\section{Deviations for non-Lipschitz functions}\label{deviations}
\setcounter{equation}{0}

Let us now turn to the interesting question on the relationship between
the distribution of a locally Lipschitz function and the distribution of
its modulus of the gradient. We still keep the setting of a metric probability 
space $(M,d,\mu)$ and assume it has a finite subgaussian constant 
$\sigma^2 = \sigma^2(\mu)$ $(\sigma \geq 0)$. 

Let us say that a continuous function $f$ on $M$ is locally Lipschitz, if 
$|\nabla f(x)|$ is finite for all $x \in M$. Recall that we consider the sets
\be \label{7.1}
A = \{x \in M: |\nabla f(x)| \leq L\}, \qquad L > 0.
\en
First we state a more general version of Corollary \ref{c.1.2}.

\begin{thm}\label{t.7.1}
Assume that a locally Lipschitz function $f$ on $M$
has Lipschitz semi-norms $\leq L$ on the sets of the form \eqref{7.1}. If
$\mu\{|\nabla f| \geq L_0\} \leq \frac{1}{2}$, then for all $t>0$,
\be  \label{7.2}
(\mu \otimes \mu) \big\{|f(x) - f(y)| \geq t\big\} \, \leq \, 2\,\inf_{L \geq L_0}
\Big[\,e^{-t^2/c\sigma^2 L^2} + \mu\big\{|\nabla f| > L\big\}\Big],
\en
where $c$ is an absolute constant.
\end{thm}

\begin{proof}
Although the argument is already mentioned in Section \ref{intro}, let us replace \eqref{1.6} with 
a slightly different bound. Applying Theorem \ref{t.1.1}, the definition \eqref{1.1} yields
\[
\int\!\!\!\int e^{t(f(x)-f(y))}\,d\mu_A(x)d\mu_A(y) \leq e^{c\sigma^2 L^2 t^2/2}, 
\quad {\rm for \ all} \ \ t \in \R,
\]
where $A$ is defined in \eqref{7.1} with $L \geq L_0$, and where $c$ is universal constant.
From this, for any $t>0$,
\[
(\mu_A \otimes \mu_A)\,\{(x,y) \in A \times A: |f(x) - f(y)| \geq t\} \leq 
2 e^{-t^2/(2c\sigma^2 L^2)},
\]
and therefore
\[
(\mu \otimes \mu)\, \{(x,y) \in A \times A: |f(x) - f(y)| \geq t\} \leq 
2 e^{-t^2/(2c\sigma^2 L^2)}.
\]
The product measure of the complement of $A \times A$ does not exceed 
$2 \mu\{|\nabla f(x)| > L\}$, and we obtain \eqref{7.2}.
\end{proof}

If $\int e^{|\nabla f|^2}\, d\mu \leq 2$, we have, by Chebyshev's inequality,
$\mu\{|\nabla f| \geq L\} \leq 2e^{-L^2}$, so one may take $L_0 = \sqrt{\log 4}$. 
Theorem \ref{t.7.1} then gives that, for any $L^2 \geq \log 4$,
\[
(\mu \otimes \mu) \big\{|f(x) - f(y)| \geq t\big\} \, \leq \, 2\,
e^{-t^2/c\sigma^2 L^2} + 4e^{-L^2}.
\]
For $t \geq 2\sigma$ one may choose here $L^2 = \frac{t}{\sigma}$, leading to
\[
(\mu \otimes \mu) \big\{|f(x) - f(y)| \geq t\big\} \, \leq \, 6\, e^{-t/c\sigma}\,,
\]
for some absolute constant $c > 1$. In case $0 \leq t \leq 2\sigma$, this inequality
is fulfilled automatically, so it holds for all $t \geq 0$. As a result,
with some absolute constant $C$,
\[
\|f(x) - f(y)\|_{\psi_1} \leq C\sigma,
\]
which is an equivalent way to state the inequality of Corollary \ref{c.1.2}.

As we have already mentioned, with the same arguments inequalities like \eqref{7.2} 
can be derived on the basis of subgaussian constants defined for different 
classes of functions. For example, one may consider the subgaussian constant
$\sigma_{\cal F}^2(\mu)$ for the class $\cal F$ of all convex Lipschitz functions 
$f$ on the Euclidean space $M = \R^n$ (which we equip with the Euclidean 
distance). Note that $|\nabla f(x)|$ is everywhere finite in the $n$-space, when
$f$ is convex. Keeping in mind  Remark \ref{r.4.3}, what we need is the following analog of 
Kirszbraun's theorem:

\begin{lem}\label{l.7.2}
Let $f$ be a convex function on $\R^n$. For any $L>0$, there 
exists a convex function $g$ on $\R^n$ such that $f = g$ on the set 
$A = \{x: |\nabla f(x)| \leq L\}$ and $|\nabla g| \leq L$ on $\R^n$.
\end{lem}

Accepting for a moment this lemma without proof, we get:

\begin{thm}\label{t.7.3}
Assume that a convex function $f$ on $\R^n$ satisfies 
$\mu\{|\nabla f| \geq L_0\} \leq \frac{1}{2}$. Then for all $t>0$,
\[
(\mu \otimes \mu) \big\{|f(x) - f(y)| \geq t\big\} \, \leq \, 2\,\inf_{L \geq L_0}
\Big[\,e^{-t^2/c\sigma^2 L^2} + \mu\big\{|\nabla f| > L\big\}\Big],
\]
where $\sigma^2 = \sigma_{\cal F}^2(\mu)$ and $c$ is an absolute constant.
\end{thm}

For illustration, let $\mu = \mu_1 \otimes \dots \otimes \mu_n$ be an arbitrary product 
probability measure on the cube $[-1,1]^n$. If $f$ is convex and Lipschitz on $\R^n$, 
thus with $|\nabla f| \leq 1$, then
\be \label{7.3}
(\mu \otimes \mu) \big\{|f(x) - f(y)| \geq t\big\} \, \leq \, 2e^{-t^2/c}.
\en
This is one of the forms of Talagrand's concentration phenomenon for the family of 
convex sets/functions (cf. \cite{T1,T2}, \cite{M}, \cite{L}). That is, the subgaussian constants 
$\sigma_{\cal F}^2(\mu)$ are bounded for the class $\cal F$ of convex Lipschitz $f$ 
and product measures $\mu$ on the cube. Hence, using Theorem \ref{t.7.3}, Talagrand's 
deviation inequality \eqref{7.3} admits a natural extension to the class of non-Lipschitz 
convex functions:

\begin{cor}\label{c.7.4}
Let $\mu$ be a product probability measure on the cube, and 
let $f$ be a convex function on $\R^n$. If $\mu\{|\nabla f| \geq L_0\} \leq \frac{1}{2}$, 
then for all $t>0$,
%\be \label{7.4}
\[
(\mu \otimes \mu) \big\{|f(x) - f(y)| \geq t\big\} \, \leq \, 2\,\inf_{L \geq L_0}
\Big[\,e^{-t^2/c L^2} + \mu\big\{|\nabla f| > L\big\}\Big],
\]
%\en
where $c$ is an absolute constant.
\end{cor}

In particular, we have a statement similar to Corollary \ref{c.1.2} -- for this family of 
functions, namely
\[
\|f - m\|_{L^{\psi_1}(\mu)} \leq c\, \| \nabla f \|_{L^{\psi_2}(\mu)},
\]
where $m$ is the $\mu$-mean of $f$.

\begin{proof}[Proof of Lemma \ref{l.7.2}.] An affine function $l_{a,v}(x) = a + \left<x,v\right>$
($v \in \R^n$, $a \in \R$) may be called to be a tangent function to $f$, if $f \geq l$ on 
$\R^n$ and $f(x) = l_{a,v}(x)$ for at least one point $x$. It is well-known that
\[
f(x) = \sup\{l_{a,v}(x): l_{a,v} \in {\cal L}\},
\]
where ${\cal L}$ denotes the collection of all tangent functions $l_{a,v}$. Put,
\[
g(x) = \sup\{l_{a,v}(x): l_{a,v} \in {\cal L}, \ |v| \leq L\}.
\]
By the construction, $g \leq f$ on $\R^n$ and, moreover,
\bee
\|g\|_{\rm Lip} 
 & \leq &
\sup\{\|l_{a,v}\|_{\rm Lip}: l_{a,v} \in {\cal L}, \ |v| \leq L\} \\
 & = &
\sup\{|v|: l_{a,v} \in {\cal L}, \ |v| \leq L\} \, \leq \, L.
\ene
It remains to show that $g = f$ on the set $A = \{|\nabla f| \leq L\}$.
Let $x \in A$ and let $l_{a,v}$ be tangent to $f$ and such that $l_{a,v}(x) = f(x)$.
This implies that $f(y) - f(x) \geq \left<y-x,v\right>$ for all $y \in \R^n$
and hence
\[
|\nabla f(x)| \, = \, \limsup_{y \rightarrow x} \frac{|f(y) - f(x)|}{|y-x|} \, \geq \,
\limsup_{y \rightarrow x} \frac{\left<y-x,v\right>}{|y-x|} \, = \, v.
\]
Thus, $|v| \leq L$, so that $g(x) \geq l_{a,v}(x) = f(x)$.
\end{proof}

\section{Optimality}\label{opt}
\setcounter{equation}{0}

Here we show that the logarithmic dependence in $\mu(A)$ in Theorems \ref{t.1.1} and \ref{t.1.3} 
is optimal, up to the universal constant $c$. We provide several examples.

\begin{ex}
Let us return to Example 4), Section \ref{examples}, of the hypercube 
$M = \{0,1\}^n$, which we equip with the Hamming distance $d_n$ and the uniform measure 
$\mu$. Let us test the inequality \eqref{6.1} of Corollary \ref{c.6.4} on the set $A \subset \{-1,1\}^n$ 
consisting of $n+1$ points
\[
(0,0,0,\dots,0), \ \ (1,0,0,\dots,0),  \ \ (1,1,0,\dots,0), \ \ \dots, \ \ (1,1,1,\dots,1).
\] 
We have $\mu(A)=(n+1)/2^n \geq 1/2^n$. The function $f:A \to \R$, defined by
\[
f(x)=\sharp\{i: x_i=1\}-\frac{n}{2},
\] 
has a Lipschitz semi-norm $\|f\|_{\rm Lip} \leq 1$ with respect to $d$ and the
$\mu_A$-mean zero. Moreover, $\int f^2\,d\mu_A = \frac{n(n+2)}{12}$. Expanding 
the inequality $\int e^{tf}\, d\mu_A \leq e^{\sigma^2(\mu_A)\, t^2/2}$ at the origin
yields $\int f^2\,d\mu_A \leq \sigma^2(\mu_A)$. Hence, recalling that 
$\sigma^2(\mu) \leq \frac{n}{4}$, we get
\bee
\sigma^2(\mu_A) \
 & \geq & 
\int f^2\,d\mu_A  \, \geq \, \frac{n^2}{12} \\
 & \geq & 
\frac{n}{3}\, \sigma^2(\mu) \, \geq \,
\frac{1}{3\log 2}\, \sigma^2(\mu)\, \log\Big(\frac{1}{\mu(A)}\Big).  
\ene
This example shows the optimality of \eqref{6.1} in the regime $\mu(A) \to 0$.
\end{ex}

\begin{ex}
Let $\gamma_n$ be the standard Gaussian measure on $\R^n$ of dimension $n \geq 2$. 
We have $\sigma^2(\gamma_n)=1$. Consider the normalized measure $\gamma_{A_R}$ on the set 
\[
A_R=\left\{(x_1,x_2,\ldots,x_n) \in \R^n: \ x_1^2+x_2^2 \geq R^2\right\}, \qquad R \geq 0.
\]
Using the property that the function $\frac{1}{2}\,(x_1^2 + x_2^2)$ 
has a standard exponential distribution under the measure $\gamma_n$,
we find that $\gamma_n(A_R) = e^{-R^2/2}$. Moreover,
\bee
s^2(\gamma_{A_R}) \, \geq \, \Var_{\gamma_{A_R}}(x_1)
 & = &
\int x_1^2\, d \gamma_{A_R}(x) \, = \, \frac{1}{2} \int (x_1^2+x_2^2)\, d \gamma_{A_R}(x) \\ 
 & = & 
\frac{1}{e^{-R^2/2}} \int_{R^2/2}^\infty r e^{-r}\, dr
 \, = \,
\frac{R^2}{2} + 1 \, = \, \log\Big(\frac{e}{\gamma_n(A_R)}\Big).
\ene
Therefore,
\[
\sigma^2(\gamma_{A_R}) \, \geq \, s^2(\gamma_{A_R}) \, \geq \, 
\log\Big(\frac{e}{\gamma_n(A_R)}\Big),
\]
showing that the inequality \eqref{1.4} of Theorem \ref{t.1.1} is optimal, up to the universal 
constant, for any value of $\gamma_n(A) \in [0,1]$.
\end{ex}

\begin{ex}
A similar conclusion can be made about the uniform probability measure $\mu$ on the
Euclidean ball $B(0,\sqrt{n})$ of radius $\sqrt{n}$, centred 
at the origin (asymptotically for growing dimension $n$). To see this, it is sufficient 
to consider the cylinders
\[
A_\ep = \big\{(x_1,y) \in \R \times \R^{n-1}: |x_1|\leq \sqrt{n-\ep^2} \ \, {\rm and} \, \,  
|y|\leq \ep\big\}, \qquad 0 < \ep \leq \sqrt{n},
\]
and the function $f(x) = x_1$. We leave to the readers corresponding computations.
\end{ex}

\begin{ex}
Let $\mu$ be the two-sided exponential measure on $\R$ with density 
$\frac{1}{2}\,e^{-|x|}$. In this case $\sigma^2(\mu) = \infty$, but, as easy to see, 
$2 \leq s^2(\mu) \leq 4$ (recall that $\lambda_1(\mu) = \frac{1}{4}$). 
We are going to test optimality of the inequality \eqref{1.7} on the sets 
$A_R = \{x \in \R: |x |\geq R\}$ ($R \geq 0$). Clearly, $\mu(A_R) = e^{-R}$, and 
we find that
\bee
s^2(\mu_{A_R}) \geq \Var_{\mu_{A_R}}(x)
 & = &
\int_{-\infty}^\infty x^2\, d\mu_{A_R}(x) \, = \, 
\frac{1}{e^{-R}} \int_R^\infty r^2 e^{-r}\, dr \\
 & = &
R^2 + 2R + 2 \, \geq \, (R + 1)^2 \, = \, \log^2\Big(\frac{e}{\mu(A_R)}\Big).
\ene
Therefore,
\[
s^2(\mu_{A_R}) \geq \log^2\Big(\frac{e}{\mu(A_R)}\Big),
\]
showing that the inequality \eqref{1.7} is optimal, up to the universal constant, for 
any value of $\mu(A) \in (0,1]$.
\end{ex}

\noindent {\bf Acknowledgment.} {The authors gratefully acknowledge the support and hospitality 
of the Institute for Mathematics \& Applications, and the University of Minnesota, 
Minneapolis, where much of this work was conducted. The second named author would 
like to acknowledge the hospitality of the Georgia Institute of Technology, Atlanta,
during the period 02/8-13/2015.}

\section*{Appendix}

\begin{proof}[Proof of Lemma \ref{l.2.1}.] Using the homogeneity, in order to derive the right-hand side inequality
in \eqref{2.1}, we may assume that $\sup_{p \geq 1} \frac{\|f\|_p}{\sqrt{p}} \leq 1$.
Then
$
\int |f|^p\,d\mu \leq p^{p/2}
$
for all $p \geq 1$, and by Chebyshev's inequality,
\[
1 - F(t) \equiv \mu\{|f| \geq t\} \leq \Big(\frac{\sqrt{p}}{t}\Big)^p, \quad
{\rm for \ all} \ \ t>0.
\]
If $t \geq 2$, choose here $p = \frac{1}{4}\,t^2$, in which case
$1 - F(t) \leq 2^{-\frac{1}{4}\,t^2}$.
Integrating by parts, we have, for any $0 < \ep < \frac{\log 2}{4}$,
\bee
\int e^{\ep f^2}\,d\mu 
 & = &
-\int_0^\infty e^{\ep t^2}\,d(1-F(t)) \\
 & = &
1 + 2\ep \int_0^2 t e^{\ep t^2}\,(1-F(t))\,dt + 
2\ep \int_2^\infty t e^{\ep t^2}\,(1-F(t))\,dt \\
 & \leq &
1 + 2\ep \int_0^2 t e^{\ep t^2}\,dt + 
2\ep \int_2^\infty t e^{\ep t^2}\,e^{-\frac{\log 2}{4}\,t^2}\,dt \\
 & = &
e^{4\ep} + \frac{\ep}{\frac{\log 2}{4} - \ep}\ e^{-(\log 2 - 4\ep)}
 \ = \ 
e^{4\ep} \bigg(1 + \frac{\ep}{2(\frac{\log 2}{4} - \ep)}\bigg).
\ene
If $\ep \leq \frac{\log 2}{8}$, the latter expression does not exceed
$\frac{3}{2}\,e^{4\ep}$ which does not exceed 2 for $\ep \leq \frac{\log(4/3)}{4}$.
Both inequalities are fulfilled for $\ep = \frac{\log 2}{10}$, and with this value
$\int e^{\ep f^2}\,d\mu \leq 2$. Hence
\[
\|f\|_{L^{\psi_2}(\mu)} \leq \frac{1}{\sqrt{\ep}} = \sqrt{\frac{10}{\log 2}} < 4,
\]
which yields the right inequality in \eqref{2.1}.
Conversely, if $\|f\|_{L^{\psi_2}(\mu)} = 1$, then $\int e^{f^2}\,d\mu = 2$. Since
$u(t) = t^p\, e^{-t^2}$ is maximized in $t>0$ at $t_0 = \sqrt{\frac{p}{2}}$, we get
\[
\|f\|_p^p = \int u(f) e^{f^2}\,d\mu \leq u(t_0) \cdot 2 = 
2\,\bigg(\frac{\sqrt{p}}{\sqrt{2e}}\bigg)^p.
\]
Hence, $\frac{\|f\|_p}{\sqrt{p}} \leq \frac{2^{1/p}}{\sqrt{2e}} < 1$, 
which yields the left inequality.

Now, let us turn to \eqref{2.2} and assume that 
$\sup_{p \geq 1} \frac{\|f\|_p}{p} = 1$. Then
$
\int |f|^p\,d\mu \leq p^p
$
for all $p \geq 1$, and by Chebyshev's inequality, for all $t > 0$,
\[
1 - F(t) \equiv \mu\{|f| \geq t\} \leq \Big(\frac{p}{t}\Big)^p.
\]
If $t \geq 2$, we may choose here $p = \frac{1}{2}\,t$ in which case
$1 - F(t) \leq 2^{-\frac{1}{2}\,t}$, while for $1 \leq t < 2$ we choose 
$p=1$, so that $1-F(t) \leq \frac{1}{t}$. Arguing as before, we have, for any 
$0 < \ep < \frac{\log 2}{2}$,
\bee
\int e^{\ep |f|}\,d\mu 
 & = &
1 + \ep \int_0^1 e^{\ep t}\,(1-F(t))\,dt +  
\ep \int_1^2 e^{\ep t}\,(1-F(t))\,dt + \ep \int_2^\infty e^{\ep t}\,(1-F(t))\,dt \\
 & \leq &
1 + \ep \int_0^1 e^{\ep t}\,dt + 
\ep \int_1^2 \frac{e^{\ep t}}{t}\,dt +
\ep \int_2^\infty e^{\ep t}\,e^{-\frac{\log 2}{2}\,t}\,dt.
\ene
The pre-last integral can be bounded by $\int_1^2 \frac{e^{2\ep}}{t}\,dt = e^{2\ep} \log 2$,
so
\[
\int e^{\ep |f|}\,d\mu  \, \leq \,
e^{\ep} + \ep e^{2\ep} \log 2 + \frac{\ep}{\frac{\log 2}{2} - \ep}\ e^{-2(\frac{\log 2}{2} - \ep)}.
\]
For $\ep = \frac{1}{6}$, the latter expression is equal to $1.98903902...$,
and thus $\int e^{\ep |f|}\,d\mu < 2$. Hence
\[
\|f\|_{L^{\psi_1}(\mu)} \leq \frac{1}{\ep} = 6.
\]
Conversely, if $\|f\|_{L^{\psi_1}(\mu)} = 1$, then $\int e^{|f|}\,d\mu = 2$. Since
$u(t) = t^p\, e^{-t}$ is maximized at $t_0 = p$, we get
\[
\|f\|_p^p = \int u(f) e^{|f|}\,d\mu \leq u(t_0) \cdot 2 = 
2\,\bigg(\frac{p}{e}\bigg)^p.
\]
Hence,
$\frac{\|f\|_p}{p} \leq \frac{2^{1/p}}{e} < 1$, which yields the left inequality.
\end{proof}

\begin{proof}[Proof of Lemma \ref{l.4.1}.]
First assume that $\|f\|_{\psi_2} = 1$, i.e.,
$\int e^{f^2}\,d\mu = 2$. The function
\[
u(t) = \log\int e^{tf}\,d\mu
\]
is smooth, convex, with $u(0) = 0$ and 
\[
u'(t) = \frac{\int f e^{tf}\,d\mu}{\int e^{tf}\,d\mu}.
\]
In particular, $u'(0) = 0$. Note that, by Jensen's inequality, 
$\int e^{tf}\,d\mu \geq 1$, so $u(t) \geq 0$. Further differentiation gives
\[
u''(t) \, = \, \frac{\int f^2 e^{tf}\,d\mu - 
\big(\int f e^{tf}\,d\mu\big)^2}{\big(\int e^{tf}\,d\mu\big)^2} \, \leq \,
\int f^2 e^{tf}\,d\mu.
\]
Using $tf \leq \frac{t^2 + f^2}{2}$ and the elementary inequality 
$x\, e^{-x/2} \leq 2e^{-1}$, we get, for $|t| \leq 1$,
\bee
\int f^2\, e^{tf}\,d\mu 
 & \leq &
\int f^2\, e^{\frac{t^2 + f^2}{2}}\,d\mu \\
 & = &
e^{t^2/2}\int f^2\,e^{f^2/2}\,d\mu \, \leq \,
e^{t^2/2}\,2e^{-1} \int e^{f^2}\,d\mu \, \leq \, 4.
\ene
Thus, $u''(t) \leq 4$, and by Taylor's formula, $u(t) \leq 2t^2$.

On the hand, for $|t| \geq 1$, by Cauchy's inequality,
\bee
\int e^{tf}\,d\mu 
 & \leq & 
\int e^{\frac{t^2 + f^2}{2}}\,d\mu \ = \ 
e^{t^2/2}\int e^{f^2/2}\,d\mu \\ 
 & \leq &
e^{t^2/2}\,\bigg(\int e^{f^2}\,d\mu\bigg)^{1/2} \, = \, \sqrt{2}\,e^{t^2/2}
 \, \leq \, e^{(1 + \log 2)\,t^2/2}.
\ene
Hence, in this case $u(t) \leq \frac{1 + \log 2}{2}\,t^2 < t^2$. Thus,
\[
\sigma_f^2 = \sup_{t \neq 0} \frac{u(t)}{t^2/2} \leq 4,
\]
proving the right inequality of Lemma \ref{l.4.1}.

For the left inequality, let $\sigma_f^2 = 1$. Then
$\int e^{tf}\,d\mu \leq e^{t^2/2}$ for all $t \in \R$, which implies
\[
1 - F(t) \equiv \mu\{|f| \geq t\} \leq 2 e^{-t^2/2}, \qquad t \geq 0.
\]
Form this, integrating by parts, we have, for any $0 < \ep < \frac{1}{2}$,
\bee
\int e^{\ep f^2}\,d\mu 
 & = &
\int_0^\infty e^{\ep t^2}\,dF(t) \, = \, -\int_0^\infty e^{\ep t^2}\,d(1-F(t)) \\
 & = &
1 + 2\ep \int_0^\infty t e^{\ep t^2}\,(1-F(t))\,dt \\
 & \leq &
1 + 4\ep \int_0^\infty t e^{\ep t^2}\,e^{-t^2/2}\,dt
 \, = \, 
1 + \frac{2\ep}{\frac{1}{2} - \ep}.
\ene
The last expression is equal to 2 for $\ep = \frac{1}{6}$, which means that
$\|f\|_{\psi_2} \leq \sqrt{6}$.
\end{proof}


\begin{thebibliography}{99}

\bibitem[A-S]{A-S} Aida, S.; Strook, D. Moment estimates derived from Poincar\'e and logarithmic
Sobolev inequalities. Math. Research Letters 1 (1994), 75--86.


\bibitem[A-B-S]{A-B-S} 
Alon, N; Boppana, R; Spencer, J. An asymptotic isoperimetric inequality.
Geometric and Functional Anal. 8 (1998), 411--436.

\bibitem[B-L]{B-L} Bakry, D.; Ledoux, M. L\'evy--Gromov's isoperimetric inequality for an infinite 
dimensional diffusion generator. Invent. Math. 123 (1996), 259--281.


\bibitem[B1]{B1} Bobkov, S. G. Localization proof of the isoperimetric Bakry-Ledoux inequality 
and some applications. Teor. Veroyatnost. i Primenen. 47 (2002), no. 2, 340--346. 
Translation in: Theory Probab. Appl. 47 (2003), no. 2, 308--314. 


\bibitem[B2]{B2} Bobkov, S. G. Perturbations in the Gaussian isoperimetric inequality.  
J. Math. Sciences (New York), 166 (2010), no. 3, 225--238. Translated from: 
Problems in Math. Analysis 45 (2010), 3--14.


\bibitem[B-G]{B-G} Bobkov, S. G.; G\"otze, F. Exponential integrability and transportation cost 
related to logarithmic Sobolev inequalities. J. Funct. Anal. 163 (1999), no. 1, pp. 1--28. 


\bibitem[B-G-H]{B-G-H} Bobkov, S. G.; Houdr\'e, C.; Tetali, P.  The subgaussian constant and 
concentration inequalities. Israel J. Math. 156 (2006), 255--283. 


\bibitem[D-M]{D-M} Ding, J. and Mossel, E.; Mixing under Monotone Censoring. Electronic Comm. Probab. 19 (2014), 1--6.



\bibitem[K]{K} Kirszbraun, M. D. \"Uber die zusammenziehende und Lipschitzsche Transformationen. 
Fund. Math. 22 (1934) 77--108.


\bibitem[L]{L} Ledoux, M. The Concentration of Measure Phenomenon. Mathematical Surveys and
Monographs 89 (2001), Amer. Math. Soc., Providence, RI.


\bibitem[Mar]{Mar} Marton, K. Bounding $\bar{d}$-distance by informational divergence: a method 
to prove measure concentration, Ann. Probab. 24 (1996), no. 2, 857--866.


\bibitem[M]{M} Maurey, B. Some deviation inequalities. Geom. Funct. Anal. 1 (1991), 188--197.


\bibitem[MS]{MS} McShane, E. J. Extension of range of functions. Bull. Amer. Math. Soc.
40 (1934), no. 12, 837--842.


\bibitem[M-W]{M-W} Mueller, C. E.; Weissler, F. B. Hypercontractivity for the heat semigroup 
         for ultraspherical polynomials and on the $n$-sphere. 
         J. Funct. Anal. 48 (1992), 252--283.


\bibitem[T1]{T1} Talagrand, M. An isoperimetric theorem on the cube and the Khinchine--Kahane
inequalities. Proc. Amer. Math. Soc. 104 (1988), 905--909.


\bibitem[T2]{T2} Talagrand,  M. Concentration of measure and isoperimetric inequalities in product
spaces. Publ. Math. I.H.E.S. 81 (1995), 73--205.






\end{thebibliography}
\end{document}